\documentclass[oneside,english]{amsart}
\usepackage[T1]{fontenc}
\usepackage[latin9]{inputenc}
\usepackage{amsthm}
\usepackage{amssymb}
\usepackage{setspace}
\usepackage[all]{xy}
\makeatletter
\numberwithin{equation}{section}
\numberwithin{figure}{section}
  \theoremstyle{plain}
  \newtheorem*{thm*}{\protect\theoremname}
\theoremstyle{plain}
\newtheorem{thm}{\protect\theoremname}
  \theoremstyle{plain}
  \newtheorem{prop}[thm]{\protect\propositionname}
  \theoremstyle{plain}
  \newtheorem{cor}[thm]{\protect\corollaryname}

\date{}
\usepackage{tensor}

\newcommand{\Ind}{\operatorname{-Ind}}
\newcommand{\ona}{\operatorname}

\newcommand{\Hom}{\operatorname{Hom}}
\newcommand{\val}{\operatorname{val}}
\newcommand{\tsr}{\tensor[^\circ]{G}{}}
\newcommand{\oX}[2]{\tensor[^\circ]{#1}{^{#2}}}
\newcommand{\nr}{\operatorname{nr}}

\newcommand{\Fr}{\operatorname{Fr}}

\newcommand{\tors}{\operatorname{torsion}}

\newcommand{\Irr}{\operatorname{Irr}}

\makeatother

\usepackage{babel}
  \providecommand{\corollaryname}{Corollary}
  \providecommand{\propositionname}{Proposition}
  \providecommand{\theoremname}{Theorem}
\providecommand{\theoremname}{Theorem}

\begin{document}

\title{Bernstein center of supercuspidal blocks}

\author{manish mishra}

\email{manish.mishra@gmail.com}

\curraddr{Im Neuenheimer Feld 288, D-69120, Heidelberg, Germany}
\begin{abstract}
Let ${\bf G}$ be a tamely ramified connected reductive group defined
over a non-archimedean local field $k.$ We show that the Bernstein
center of a tame supercuspidal block of ${\bf G}(k)$ is isomorphic
to the Bernstein center of a depth zero supercuspidal block of ${\bf G}^{0}(k)$
for some twisted Levi subgroup of ${\bf G}^{0}$ of ${\bf G}$. 
\end{abstract}

\maketitle

\section{Introduction}

Let ${\bf G}$ be a connected reductive group defined over a non archimedean
local field $k$. Assume that ${\bf G}$ splits over a tamely ramified
extension $k^{t}$ of $k$. We will denote the group of $k$-rational
points of ${\bf G}$ by $G$ and likewise for other algebraic groups.
In \cite{Yu01}, Jiu-Kang Yu gives a very general construction of
a class of supercuspidal representations of $G$ which he calls \textit{tame}.
A tame supercuspidal representation $\pi=\pi_{\Sigma}$ of $G$ is
constructed out of a depth zero supercuspidal representation $\pi_{0}$
of $G^{0}$ and some additional data, where ${\bf G}^{0}$ is a \textit{twisted}
Levi subgroup of ${\bf G}$. By twisted, we mean that ${\bf G}^{0}\otimes k^{t}$
is a Levi factor of a parabolic subgroup of ${\bf G}\otimes k^{t}$.
The additional data, together with ${\bf G}^{0}$ and $\pi_{0}$ is
what we are denoting by $\Sigma$ in the notation $\pi_{\Sigma}$.
In \cite{Kim07}, Kim showed that under certain hypothesis, which
are met for instance when the residue characteristic is large, these
tame supercuspidals exhaust all the supercuspidals of $G$. 

The depth zero supercuspidal $\pi_{0}$ of $G^{0}$ is compactly induced
from $(K^{0},\varrho_{0})$ where $K^{0}$ is a compact mod center
open subgroup of $G^{0}$ and $\varrho_{0}$ is a representation of
$K^{0}$. The constructed representation $\pi_{\Sigma}$ is compactly
induced from $(K,\varrho)$, where $K$ is a compact mod center open
subgroup of $G$ containing $K^{0}$ and $\varrho$ is a representation
of $K$. The representation $\varrho$ is of the form $\varrho_{0}\otimes\kappa$,
where $\varrho_{0}$ is seen as a representation of $K$ by extending
from $K^{0}$ ``trivially'' (see \cite[Sec. 4]{Yu01}) and $\kappa$
is a representation of $K$ constructed out of the part of $\Sigma$
which is independent of $\varrho_{0}$. 

Let $\mathfrak{Z}^{\pi}$ (resp. $\mathfrak{Z}_{0}^{\pi_{0}}$) denote
the \textit{Bernstein center} of the\textit{ Bernstein block} (see
Section \ref{sec:Bernstein} for these terms) of $G$ (resp. $G^{0}$)
containing $\pi$ (resp. $\pi_{0}$). Under certain hypothesis $C(\overrightarrow{{\bf G}})$
\cite[Page 47]{HM08}, we show that:
\begin{thm*}
\textup{$\mathfrak{Z}^{\pi}\cong\mathfrak{Z}_{0}^{\pi_{0}}$. }Thus,
the Bernstein center of a tame supercuspidal block of $G$ is isomorphic
to the Bernstein center of a depth zero supercuspidal block of a twisted
Levi subgroup of $G$.
\end{thm*}
Let $\mathcal{H}(G,\oX{\varrho}{})$ (resp. $\mathcal{H}(G^{0},\oX{\varrho}{}_{0})$)
denote the Hecke algebra of the type constructed out of $(K,\varrho)$
(resp. $(K^{0},\varrho_{0})$) (see Sec. \ref{sub:Hecke-algebra}).
As a consequence of the above theorem, we obtain 
\[
Z(\mathcal{H}(G,\oX{\varrho}{}))\cong Z(\mathcal{H}(G^{0},\oX{\varrho}{}_{0})),
\]
 where $Z(\mathcal{H}(G,\oX{\varrho}{}))$ (resp. $Z(\mathcal{H}(G^{0},\oX{\varrho}{}_{0}))$)
denotes the center of $\mathcal{H}(G,\oX{\varrho}{})$ (resp. $\mathcal{H}(G^{0},\oX{\varrho}{}_{0})$). 

In \cite[Conj. 0.2]{Yu01}, Yu conjectures that $\mathcal{H}(G,\oX{\varrho}{})\cong\mathcal{H}(G^{0},\oX{\varrho}{}_{0})$.
This is a special case of his more general conjecture \cite[Conj. 17.7]{Yu01}.
Assuming certain conditions on $\pi_{\Sigma}$ (\cite[Sec. 5.5]{BK98})
which are satisfied quite often, for instance whenever $\pi_{\Sigma}$
is generic, in Theorem \ref{cor:yu-conj} we show that 
\[
\mathcal{H}(G,\oX{\varrho}{})\cong\mathcal{H}(G^{0},\oX{\varrho}{}_{0}).
\]

\section{\label{sec:Notations}Notations}

Throughout this article, $k$ denotes a non-archimedean local field.
For an algebraic group ${\bf G}$ defined over $k$, we will denote
its $k$-rational points by $G$. We will follow standard abuses of
notation and terminology and refer, for example, to parabolic subgroups
of $G$ in place of $k$-points of $k$-parabolic subgroups of ${\bf G}$.
Center of ${\bf G}$ will be denoted by ${\bf Z}_{{\bf G}}$. The
category of smooth representations of $G$ will be denoted by $\mathfrak{R}(G)$.
If $K$ is a subgroup of $G$ and $g\in G$, we denote $gKg^{-1}$
by $\tensor[^{g}]{K}{}$. If $\rho$ is a complex representation of
$K$, $\tensor[^{g}]{\rho}{}$ denotes the representation $x\mapsto\rho(g^{-1}xg)$
of $\tensor[^{g}]{K}{}$. For $g\in G$, we say that $g$ \textit{intertwines}
$\rho$ if the vector space $\operatorname{Hom}_{\tensor[^{g}]{K}{}\cap K}(\tensor[^{g}]{\rho}{},\rho)$
is non-zero.

\section{\label{sec:Yu}Yu's construction \cite{Yu01}}

Let ${\bf G}$ be a connected reductive group defined over a non-archimedean
local field $k$. A twisted $k$-Levi subgroup ${\bf G}^{\prime}$
of ${\bf G}$ is a reductive $k$-subgroup such that ${\bf G}^{\prime}\otimes_{k}\bar{k}$
is a Levi subgroup of ${\bf G}\otimes_{k}\bar{k}$. Yu's construction
involves the notion of a generic ${\bf G}$-datum. It is a quintuple
$\Sigma=(\overrightarrow{{\bf G}},y,\overrightarrow{r},\overrightarrow{\phi},\rho)$
satisfying the following:
\begin{enumerate}
\item $\overrightarrow{{\bf G}}=({\bf G}^{0}\subsetneq{\bf G}^{1}\subsetneq\ldots\subsetneq{\bf G}^{d}={\bf G})$
is a tamely ramified twisted Levi sequence such that ${\bf Z}_{{\bf G}^{0}}/{\bf Z}_{{\bf G}}$
is anisotropic. 
\item $y$ is a point in the extended Bruhat-Tits building of ${\bf G}^{0}$
over $k$. 
\item $\overrightarrow{r}=(r_{0},r_{1},\cdots,r_{d-1},r_{d})$ is a sequence
of positive real numbers with $0<r_{0}<\cdots<r_{d-2}<r_{d-1}\leq r_{d}$
if $d>0$, $0\leq r_{0}$ if $d=0$.
\item $\overrightarrow{\phi}=(\phi_{0},\cdots,\phi_{d})$ is a sequence
of quasi-characters, where $\phi_{i}$ is a $G^{i+1}$-\textit{generic
quasi-character} \cite[ Sec. 9]{Yu01} of $G^{i}$; $\phi_{i}$ is
trivial on $G_{y,r_{i}+}^{i}$, but nontrivial on $G_{y,r_{i}}^{i}$
for $0\leq i\leq d-1$. If $r_{d-1}<r_{d}$, $\phi_{d}$ is nontrivial
on $G_{y,r_{d}}^{i}$ and trivial on $G_{y,r_{d}+}^{d}$. Otherwise,
$\phi_{d}=1$. Here $G_{y,r}^{i}$ denote the filtration subgroups
of the parahoric at $y$ defined by Moy-Prasad (see \cite[Sec. 2.6]{MP94}). 
\item $\rho$ is an irreducible representation of $G_{[y]}^{0}$, the stabilizer
in $G^{0}$ of the image $[y]$ of $y$ in the reduced building of
${\bf G}^{0}$, such that $\rho|G_{y,0+}^{0}$ is isotrivial and $c\Ind_{G_{[y]}^{0}}^{G^{0}}\rho$
is irreducible and supercuspidal. 
\end{enumerate}
Let $K^{0}=G_{[y]}^{0}$, $K^{0+}=G_{y,0+}^{0}$, $K^{i}=G_{[y]}^{0}G_{y,s_{0}}^{1}\cdots G_{y,s_{i-1}}^{i}$
and $K^{i+}=G_{[y]}^{0}G_{y,s_{0}+}^{1}\cdots G_{y,s_{i-1}+}^{i}$
where $s_{j}=r_{j}/2$ for $i=1,\ldots,d$. In \cite[Sec. 11]{Yu01},
Yu constructs certain representation $\kappa$ of $K^{d}=K^{d}(\Sigma)$
which is independent of $\rho$ and constructed only out of $(\overrightarrow{{\bf G}},y,\overrightarrow{r},\overrightarrow{\phi})$.
He defines certain subgroups $J^{i}=({\bf G}^{i-1},{\bf G}^{i})(k)_{y,(r_{i-1},s_{i-1})}$
and $J^{i+}=({\bf G}^{i-1},{\bf G}^{i})(k)_{y,(r_{i-1},s_{i-1}+)}$
for $1\leq i\leq d$ (see \cite[Sec. 3]{Yu01} for the meaning of
notations used here). For these groups, one has 
\[
K^{i-1}J^{i}=K^{i},\qquad K^{(i-1)+}J^{i+}=K^{i+}.
\]
Also, $K^{i-1}\cap J^{i}\subset K^{(i-1)+}$. Since $\rho$ is iso-trivial
on $K^{0+}$, one can successively inflate the representation $\rho$
of $K^{0}$ to a representation of $K^{d}$, which we again denote
by $\rho$, via the maps 
\[
K^{i}\twoheadrightarrow K^{i-1}J^{i}/J^{i}=K^{i-1}/(K^{i-1}\cap J^{i})
\]
(see \cite[Sec. 4]{Yu01} for details). Write $\rho_{\Sigma}:=\rho\otimes\kappa$. 
\begin{thm}
[Yu]$\pi_{\Sigma}:=c\Ind_{K^{d}}^{G}\rho_{\Sigma}$ is irreducible
and thus supercuspidal. 
\end{thm}
The following theorem of Kim \cite{Kim07} says that under certain
hypothesis (which are met for instance when the residue characteristic
is sufficiently large), the representations $\pi_{\Sigma}$ for various
generic ${\bf G}$-datum $\Sigma$ exhaust all the supercuspidal representations
of $G$. 
\begin{thm}
[Ju-Lee Kim]\label{thm:Kim}Suppose the hypothesis $(\ona{H}k)$,
$(\ona{HB})$, $(\ona{HGT})$ and $(\ona{H}\mathcal{N})$ in \cite{Kim07}
are valid. Then all the supercuspidal representations of $G$ arise
through Yu's construction. 
\end{thm}
In \cite[Theorem 6.6, 6.7]{HM08} under certain hypothesis denoted
by $C(\overrightarrow{{\bf G}})$ \cite[Page 47]{HM08}, Hakim and
Murnaghan determine when two supercuspidal representations are equivalent:
\begin{thm}
[Hakim-Murnaghan]\label{thm:HM}Let $\Sigma=(\overrightarrow{{\bf G}},y,\overrightarrow{r},\overrightarrow{\phi},\rho)$
and $\Sigma^{\prime}=(\overrightarrow{{\bf G}}^{\prime},y^{\prime},\overrightarrow{r}^{\prime},\overrightarrow{\phi}^{\prime},\rho^{\prime})$
be two generic $G$-data. Set $\phi=\Pi_{i=1}^{d}\phi_{i}|G^{0}$,
$\phi^{\prime}=\Pi_{i=1}^{d^{\prime}}\phi_{i}^{\prime}|G^{0\prime}$,
$\pi_{0}=c\Ind_{G_{[y]}^{0}}^{G^{0}}\rho$ and $\pi_{0}^{\prime}=c\Ind_{G_{[y^{\prime}]}^{0^{\prime}}}^{G^{0\prime}}\rho^{\prime}$.
Then $\pi_{\Sigma}\cong\pi_{\Sigma^{\prime}}$ if and only if there
exists $g\in G$ such that $\tensor[]{K^{d}(\Sigma)}{}=\tensor[^{g}]{K^{d^{\prime}}(\Sigma^{\prime})}{}$
and $\rho_{\Sigma}=\tensor[^{g}]{\rho}{_{\Sigma^{\prime}}}$ if and
only if $G^{0}={}^{g}G^{0\prime}$ and $\pi_{0}\otimes\phi\cong\tensor[^{g}]{(\pi_{0}^{\prime}\otimes\phi^{\prime})}{}$. 
\end{thm}

\section{\label{sec:Bernstein}Bernstein center}

\subsection{Bernstein decomposition}

Let $X_{k}({\bf G})=\Hom({\bf G},\mathbb{G}_{m})$, the lattice of
$k$-rational characters of ${\bf G}$. Let 
\[
\tensor[^{\circ}]{G}{}:=\{g\in G:\val_{k}(\chi(g))=0,\forall\chi\in X_{k}({\bf G})\}.
\]

In \cite[Section 7]{Kott97}, Kottwitz defined a functorial homomorphism
$\kappa_{G}^{\prime}:G\twoheadrightarrow X_{*}({\bf Z}_{{\bf G}})_{I_{k}}^{\Fr}$.
Here $X_{*}({\bf Z}_{{\bf G}})$ denotes the co-character lattice
of ${\bf Z}_{{\bf G}}$, $(\cdot)^{\Fr}$ (resp. $(\cdot)_{I_{k}}$)
denotes taking invariant (resp. coinvariant) with respect to Frobenius
$\Fr$ (resp. inertia subgroup $I_{k}$). The map $\kappa_{G}^{\prime}$
induces a functorial surjective map: 
\begin{equation}
\kappa_{G}:G\twoheadrightarrow X_{*}({\bf Z}_{{\bf G}})_{I_{k}}^{\Fr}/\tors\label{eq:kott}
\end{equation}
 and $\ker(\kappa_{G})$ is precisely $\tsr$ (see \cite[Sec. 3.3.1]{Hai11}).

Let $X_{\nr}(G):=\Hom(G/\tsr,\mathbb{C}^{\times})$ denote the group
of \textit{unramified characters} of $G$. For a smooth representation
$\pi$ of $G$, the representations $\pi\otimes\chi,$ $\chi\in X_{\nr}(G(k))$
are called the \textit{unramified twists} of $\pi$.

Consider the collection of all cuspidal pairs $(L,\sigma)$ consisting
of a Levi subgroup $L$ of $G$ and an irreducible cuspidal representation
$\sigma$ of $L$. Define an equivalence relation $\sim$ on the class
of all cuspidal pairs by 
\[
(L,\sigma)\sim(M,\tau)\mbox{ if }\tensor[^{g}]{L}{}=M\mbox{ and }\tensor[^{g}]{\sigma}{}\cong\tau\nu,
\]
 for some $g\in G$ and some $\nu\in X_{\nr}(M)$. Write $[L,\sigma]_{G}$
for the equivalence class of $(L,\sigma)$ and $\mathfrak{B}(G)$
for the set of all equivalence classes. The set $\mathfrak{B}(G)$
is called the \textit{Bernstein spectrum }of $G$. We say that a smooth
irreducible representation $\pi$ has \textit{inertial support} $s:=[L,\sigma]_{G}$
if $\pi$ appears as a subquotient of a representation parabolically
induced from some element of $\mathfrak{s}$. Define a full subcategory
$\mathfrak{R}^{\mathfrak{}}(G)^{\mathfrak{s}}$ of $\mathfrak{R}(G)$
as follows: a smooth representation $\pi$ belongs to $\mathfrak{R}(G)^{\mathfrak{s}}$
iff each irreducible subquotient of $\pi$ has inertial support $\mathfrak{s}$.
The categories $\mathfrak{R}^{\mathfrak{}}(G)^{\mathfrak{s}},\mathfrak{s\in\mathfrak{B}}(G)$,
are called the \textit{Bernstein Blocks} of $G$.
\begin{thm}
[Bernstein]\label{thm:B-De}We have 
\[
\mathfrak{R}(G)=\prod_{\mathfrak{s\in\mathfrak{B}}(G)}\mathfrak{R}^{\mathfrak{}}(G)^{\mathfrak{s}}.
\]

\end{thm}

\subsection{\label{sub:Hecke-algebra}Hecke algebra}

Let $J$ be a compact open subgroup of $G$ and let $(\tau,W)$ be
an irreducible representation of $J$. We call $(J,\tau)$ a \textit{compact
open datum}. The \textit{Hecke algebra} $\mathcal{H}(G,\tau)$ associated
to a compact open datum $(J,\tau)$ is the space of compactly supported
functions $f:G\rightarrow\ona{End}_{\mathbb{C}}(W)$ such that 
\[
f(j_{1}gj_{2})=\tau(j_{1})f(g)\tau(j_{2}),\qquad j_{1,}j_{2}\in J\mbox{ and }g\in G.
\]
 The standard convolution operation gives $\mathcal{H}(G,\tau)$ the
structure of an associative $\mathbb{C}$-algebra with identity. 

Let $\mathfrak{R}_{\tau}(G)$ be the subcategory of $\mathfrak{R}(G)$
consisting of smooth representations which are generated by their
$\tau$- isotypic component. If $\mathfrak{R}_{\tau}(G)=\mathfrak{R}(G)^{\mathfrak{s}}$
for some $\mathfrak{s}\in\mathfrak{B}(G)$, then we say that $(J,\tau)$
is an $\mathfrak{s}$-\textit{type}. Let $\mathcal{H}(G,\tau)-\mathfrak{Mod}$
denote the category of non-degenerate $\mathcal{H}(G,\tau)$ modules.
If $(J,\tau)$ is an $\mathfrak{s}$-type, then $\mathfrak{R}(G)^{\mathfrak{s}}$
is equivalent to $\mathcal{H}(G,\tau)-\mathfrak{Mod}$.

\subsection{\label{sub:center}The center of $\mathfrak{R}(G)$}

Let $\mathcal{C}$ be an abelian category. The set $\ona{End}_{C}(\ona{id})$
of natural transformations of the identity functor of $\mathcal{C}$
is a ring which by definition is the \textit{center} of $\mathcal{C}$.
Denote it by $\mathfrak{Z}(\mathcal{C})$. Explicitly, $z\in\mathfrak{Z}(\mathcal{C})$
is a collection of morphisms $z_{A}:A\rightarrow A$, one for each
object $A$ in $\mathcal{C}$, such that for any morphism $f:B\rightarrow C$,
the diagram
\[
\xymatrix{B\ar@{->}[r]^{f}\ar@{->}[d]^{z_{B}} & C\ar@{->}[d]^{z_{C}}\\
B\ar@{->}[r]^{f} & C
}
\]
 commutes. 

Let $R$ be a ring with identity. Let $\mathfrak{Z}(R)$ (resp. $\mathfrak{Z}(R-\mathfrak{Mod})$)
denote the center of $R$ (resp. the center of the category of left
$R$-modules). There is a canonical ring isomorphism
\begin{equation}
c\in\mathfrak{Z}(R)\mapsto\mu_{c}\in\mathfrak{Z}(R-\mathfrak{Mod}),\label{eq:ring-isom}
\end{equation}
 where $\mu_{_{c}}$ acts on each left $R$-module $M$ by $\mu_{c}(m)=cm$,
for all $m\in M$ (see \cite[Sec. 1.6.2]{Roche09}).

Let $\mathfrak{s}\in\mathfrak{B}(G)$. The center of $\mathfrak{Z}(G)$
(resp. $\mathfrak{Z}(G)^{\mathfrak{s}}$) of the category $\mathfrak{R}(G)$
(resp. $\mathfrak{R}(G)^{\mathfrak{s}}$) is called the \textit{Bernstein
center}. If $(J,\tau)$ is an $\mathfrak{s}$-type, then $\mathfrak{R}(G)^{\mathfrak{s}}\cong\mathcal{H}(G,\tau)-\mathfrak{Mod}$,
and therefore by Equation (\ref{eq:ring-isom}), there is a canonical
isomorphism
\begin{equation}
\mathfrak{Z}(G)^{\mathfrak{s}}\cong Z(\mathcal{H}(G,\tau)),\label{eq:center=00003DHecke}
\end{equation}
 where $Z(\mathcal{H}(G,\tau))$ denotes the center of $\mathcal{H}(G,\tau)$.

\section{\label{sec:sup-block}Supercuspidal block}

Let ${\bf G}$ be a connected reductive group over $k$. Let $\pi$
be an irreducible supercuspidal representation of $G$ of the form
$\pi=c\Ind_{J}^{G}(\tau)$, where $J$ is an open, compact mod center
subgroup of $G$ and $\tau$ is a representation of $J$. Write $\oX{J}{}=J\cap\oX{G}{}$
and let $\oX{\tau}{}$ be some irreducible component of $\tau|\oX{J}{}$.
Then 
\begin{prop}
\cite[Sec. 5.4]{BK98}The group $\oX{J}{}$ is the unique maximal
compact subgroup of $J$ and $(\oX{J}{},\oX{\tau}{})$ is a $[G,\pi]_{G}$-type
in $G$. 
\end{prop}

\subsection{\label{Comm-condition}Commutativity conditions}

Assume that the representation $\pi$ satisfies the following conditions:
\begin{enumerate}
\item The representation $\tau|\oX{J}{}$ is irreducible, i.e., $\oX{\tau}{}=\tau|\oX{J}{}$.
\item Any $g\in G$ which intertwines the representation $\oX{\tau}{}$
lies in $J$. 
\end{enumerate}
These conditions are quite frequently satisfied (see \cite[Sec. 5.5]{BK98}),
for instance if $\pi$ admits a Whittaker model (\cite[Remark 1.6.1.3]{Roche09}).
Under these assumptions, we have:
\begin{prop}
\cite[Sec. 5.5]{BK98} \label{prop:comm}The Hecke algebra $\mathcal{H}(G,\oX{\tau}{})$
associate to the type $(\oX{J}{},\oX{\tau}{})$ is commutative. 
\end{prop}

\section{\label{sec:Main}Main result}

We use the notations of Section \ref{sec:Yu}. Fix a generic ${\bf G}$-datum
$\Sigma=(\overrightarrow{{\bf G}},y,\overrightarrow{r},\overrightarrow{\phi},\rho)$.
Then $\pi_{\Sigma}:=c\Ind_{K^{d}}^{G}\rho_{\Sigma}$ is an irreducible
supercuspidal representation of $G$, where $\rho_{\Sigma}$ is of
the form $\rho\otimes\kappa$ and $\kappa$ is a representation of
$K^{d}$, constructed only out of the data $(\overrightarrow{{\bf G}},y,\overrightarrow{r},\overrightarrow{\phi})$.
The representation $\pi_{0}=c\Ind_{G_{[y]}^{0}}^{G^{0}}\rho$ of $G^{0}$
is depth zero supercuspidal. Set $\mathfrak{s:=}[G,\pi_{\Sigma}]_{G}$
and $\mathfrak{s}_{0}:=[G^{0},\pi_{0}]_{G^{0}}$. Let $\mathfrak{Z}(G)$
(resp. $\mathfrak{Z}(G)^{\mathfrak{s}}$, resp. $\mathfrak{Z}(G^{0})^{\mathfrak{s}_{0}}$)
be the Bernstein center of the category $\mathfrak{R}(G)$ (resp.
$\mathfrak{R}(G)^{\mathfrak{s}}$, resp. $\mathfrak{R}(G^{0})^{\mathfrak{s}_{0}}$).
Let $\Irr^{\mathfrak{s}}(G)$ (resp. $\Irr^{\mathfrak{s_{0}}}(G^{0})$)
denote the isomorphism classes of irreducible elements in $\mathfrak{R}(G)^{\mathfrak{s}}$
(resp. $\mathfrak{R}(G^{0})^{\mathfrak{s_{0}}}$). We assume the hypothesis
$C(\overrightarrow{{\bf G}})$ in \cite[Page 47]{HM08} in the rest
of this section. 

By functoriality of the map (\ref{eq:kott}), the inclusion ${\bf G}^{0}\hookrightarrow{\bf G}$
induces a map 
\begin{equation}
\chi\in X_{\nr}(G)\mapsto\chi|G^{0}\in X_{\nr}(G^{0}).\label{eq:rest-1}
\end{equation}

\begin{thm}
\label{thm:main}The map $\mathfrak{f}_{\Sigma}:\pi_{\Sigma}\otimes\nu\in\Irr^{\mathfrak{s}}(G)\mapsto\pi_{0}\otimes(\nu|G^{0})\in\Irr^{\mathfrak{s_{0}}}(G^{0})$,
$\nu\in X_{\nr}(G)$, is well defined and is a bijection. Consequently,
there is an isomorphism $f_{\Sigma}:\mathfrak{Z}(G)^{\mathfrak{s}}\cong\mathfrak{Z}(G^{0})^{\mathfrak{s}_{0}}$.\end{thm}
\begin{proof}
We first prove well definedness. Suppose $\pi_{\Sigma}\otimes\chi\cong\pi_{\Sigma}$
for $\chi\in X_{\nr}(G)$. Then we want to show that $\pi_{0}\otimes\chi|G^{0}\cong\pi_{0}$.
Define a new quintuple $\Sigma_{\chi}=(\overrightarrow{{\bf G}},y,\overrightarrow{r},\overrightarrow{\phi},\rho\otimes(\chi|K^{0}))$.
We have $\pi_{\Sigma}\otimes\chi\cong c\Ind_{K^{d}}^{G}(\rho\otimes\kappa\otimes(\chi|K^{d}))$.
Since $\chi$ is unramified, it follows that $\pi_{\Sigma}\otimes\chi\cong\pi_{\Sigma_{\chi}}$.
By Theorem \ref{thm:HM}, $\pi_{\Sigma}\cong\pi_{\Sigma_{\chi}}$
is equivalent to $(K^{d},\rho_{\Sigma})$ being conjugate to $(K^{d},\rho_{\Sigma_{\chi}})$
by an element $g\in G$. Since $\rho_{\Sigma}|K^{d+}=\rho_{\Sigma_{\chi}}|K^{d+}$,
it follows that $g$ intertwines $\rho_{\Sigma}|K^{d+}$. By \cite[Prop. 4.4 and 4.1]{Yu01},
it implies that $g\in K^{d}G^{0}K^{d}$. Thus we can assume without
loss of generality that $g\in G^{0}$. Let $\rho^{\prime}=(\rho\otimes\chi|K^{0})$.
Then by Theorem \ref{thm:HM}, we get $\pi_{0}\otimes\phi\cong\tensor[]{(\pi_{0}^{\prime}\otimes\phi)}{}$
as $G^{0}$-representations, where $\phi$ is as in Theorem \ref{thm:HM}
and $\pi_{0}^{\prime}:=c\Ind_{G_{[y]}^{0}}^{G^{0}}\rho^{\prime}\cong\pi_{0}\otimes(\chi|G^{0})$.
It follows that $\pi_{0}\otimes\chi|G^{0}\cong\pi_{0}$ and therefore
$\mathfrak{f}_{\Sigma}$ is well defined. Now if $\chi\in X_{\nr}(G)$
is such that $\pi_{0}\otimes\chi|G^{0}\cong\pi_{0}$, then it follows
from Theorem \ref{thm:HM} or directly that $\pi_{\Sigma_{\chi}}\cong\pi_{\Sigma}$,
i.e., $\pi_{\Sigma}\otimes\chi\cong\pi_{\Sigma}$. This shows that
the map $\mathfrak{f}_{\Sigma}$ is also injective. 

We now prove surjectivity. Now given $\nu\in X_{\nr}(G^{0})$, using
notation similar to before, write $\Sigma_{\nu}=(\overrightarrow{{\bf G}},y,\overrightarrow{r},\overrightarrow{\phi},\rho\otimes(\nu|K^{0}))$.
Let $(\oX{K}{d},\oX{\rho}{}_{\Sigma})$ (resp. $(\oX{K}{d},\oX{\rho}{}_{\Sigma_{\nu}})$)
be a type constructed out of $(K^{d},\rho_{\Sigma})$ (resp. $(K^{d},\rho_{\Sigma_{\nu}})$
as in Sec. \ref{sec:sup-block}. Then $\oX{K}{0}=G_{y}^{0}$ and $\oX{K}{d}=\oX{K}{0}G_{y,s_{0}}^{1}\cdots G_{y,s_{d-1}}^{d}$
(see notations in Sec. \ref{sec:Yu}) is the maximal compact subgroup
of $K^{d}$ (see \cite[Cor. 15.3]{Yu01}). Since $\rho_{\Sigma}|\oX{K}{d}=\rho_{\Sigma_{\chi}}|\oX{K}{d}$,
we can assume that $(\oX{K}{d},\oX{\rho}{}_{\Sigma})=(\oX{K}{d},\oX{\rho}{}_{\Sigma_{\nu}})$.
Now since $(\oX{K}{d},\oX{\rho}{}_{\Sigma})$ is an $\mathfrak{s}$-type,
it follows that $\pi_{\Sigma_{\nu}}\cong\pi_{\Sigma}\otimes\chi$
for some $\chi\in X_{\nr}(G)$. By the argument used in the proof
of the well definedness of the map $\mathfrak{f}_{\Sigma}$ in the
previous paragraph, we get, $\pi_{0}\otimes\nu\cong\pi_{0}\otimes(\chi|G^{0})$,
i.e., $\pi_{0}\otimes\nu$ is the image of $\pi_{\Sigma}\otimes\chi$
under $\mathfrak{f}_{\Sigma}$. Thus $\mathfrak{f}_{\Sigma}$ is also
surjective.

We thus have a bijection $\mathfrak{f}_{\Sigma}:\pi_{\Sigma}\otimes\chi\in\Irr^{\mathfrak{s}}(G)\mapsto\pi_{0}\otimes(\chi|G^{0})\in\Irr^{\mathfrak{s_{0}}}(G^{0}),\chi\in X_{\nr}(G)$.
Since $\mathfrak{Z}(G)^{\mathfrak{s}}$ (resp. $\mathfrak{Z}(G^{0})^{\mathfrak{s}_{0}}$)
is canonically the ring of regular functions on $\Irr^{\mathfrak{s}}(G)$
(resp. $\Irr^{\mathfrak{s_{0}}}(G^{0})$) \cite[Prop. 1.6.4.1]{Roche09},
the Theorem follows. 
\end{proof}
For each irreducible object $\tau\in\mathfrak{R}(G)$ and $z\in\mathfrak{Z}(G)$,
denote by $\chi_{z}(\tau)$, the scalar by which $z$ acts on $\tau$. 
\begin{cor}
Let $z\in\mathfrak{Z}(G)^{\mathfrak{s}}$ and $\pi\in\Irr^{\mathfrak{s}}(G)$.
Then $\chi_{z}(\pi)=\chi_{f_{\Sigma}(z)}(\mathfrak{f}_{\Sigma}(\pi)).$ \end{cor}
\begin{proof}
This follows from \cite[Prop. 1.6.4.1]{Roche09} and Theorem \ref{thm:main}.
\end{proof}
For an algebra $\mathcal{A}$, denote by $Z(\mathcal{A})$ the center
of $\mathcal{A}.$ Let $\mathcal{H}(G,\oX{\rho}{}_{\Sigma})$ (resp.
$\mathcal{H}(G^{0},\oX{\rho}{})$) denote the Hecke algebra associated
to the compact open data $({}^{\circ}K^{d},\oX{\rho}{}_{\Sigma})$
(resp. $({}^{\circ}K^{0},\oX{\rho}{})$) (see Sec. \ref{sub:Hecke-algebra}). 
\begin{cor}
\label{cor:hec_cent} $Z(\mathcal{H}(G,\oX{\rho}{}_{\Sigma}))\cong Z(\mathcal{H}(G^{0},\oX{\rho}{}))$.\end{cor}
\begin{proof}
Since $(\oX{K}{d},\oX{\rho}{}_{\Sigma})$ (resp. $({}^{\circ}K^{0},\oX{\rho}{})$)
is an $\mathfrak{s}$-type (resp. $\mathfrak{s_{0}}$-type), this
follows from Equation \ref{eq:center=00003DHecke} and Theorem \ref{thm:main}.
\end{proof}
Now suppose that $\pi_{\Sigma}$ satisfies the commutativity conditions
of Sec. \ref{Comm-condition}. With this assumption, we get the following
result.
\begin{thm}
\label{cor:yu-conj}$\mathcal{H}(G,\oX{\rho}{}_{\Sigma})\cong\mathcal{H}(G^{0},\oX{\rho}{})$.\end{thm}
\begin{proof}
By assumption, $\rho_{\Sigma}|\oX{K}{d}$ is irreducible. Since $\rho_{\Sigma}=\rho\otimes\kappa$
in the notations of Sec. \ref{sec:Yu}, this implies that $\rho|\oX{K}{0}$
is also irreducible. Now by \cite[Corr. 15.5]{Yu01}, $g\in G^{0}$
intertwines $\oX{\rho}{}$ iff it intertwines $\oX{\rho}{}_{\Sigma}$
. But then by assumption, $g\in K^{d}$. Thus any $g\in G^{0}$ which
intertwines $\oX{\rho}{}$ lies in $K^{0}$. This means that $\pi_{0}$
also satisfies the commutativity conditions of Sec. \ref{Comm-condition}.
Then by Proposition \ref{prop:comm}, the Hecke algebras $\mathcal{H}(G,\oX{\rho}{}_{\Sigma})$
and $\mathcal{H}(G^{0},\oX{\rho}{})$ are commutative. The Theorem
then follows from Corollary \ref{cor:hec_cent}.
\end{proof}

\section*{Acknowledgement}

The author is very thankful to Rainer Weissauer, Jeff Adler, Sandeep
Varma and David Kazhdan for many helpful interactions. He is especially
grateful to Jiu-Kang Yu for his careful proof reading and making important
suggestions. He is also thankful to the Math Institute at Heidelberg
University for supporting his stay during which this work was written. 

\bibliographystyle{plain}
\bibliography{sup_block}

\end{document}